\documentclass[11pt]{article}

\usepackage[reqno]{amsmath}  
\usepackage{amsfonts}
\usepackage{amssymb}
\usepackage{amsthm}
\usepackage{empheq}
\usepackage{wrapfig}            
\usepackage{epsf}               
\usepackage{graphicx}           
\usepackage{color}
\usepackage[twocolumn]{multicol}
\usepackage{indentfirst}
\usepackage[hypertex]{hyperref}

 \newtheorem{thm}{\indent Theorem}[section]
 \newtheorem{cor}{\indent Corollary}[section]
 \newtheorem{lem}{\indent Lemma}[section]
 
 \theoremstyle{definition}
 
 \newtheorem{rem}{\indent Remark}[section]
 
 \numberwithin{equation}{section}

\def\f{\frac}

\def\e{\eqref}

\def\lab{\label}

\def\vph{\varphi}
\def\i13{(i=1,2,3)}
\def\2i4{(i=2,3,4)}
\def\ij1n{(i,j=1,\cdots,n)}
\def\dis{\displaystyle}

\allowdisplaybreaks

\newcommand{\beq}{\begin{equation}}
\newcommand{\eeq}{\end{equation}}
\newcommand{\bs}{\begin{split}}
\newcommand{\es}{\end{split}}

\begin{document}

\title{\bf Exact boundary observability for nonautonomous quasilinear wave equations
}
\author{
Lina Guo\thanks{School of Mathematical Sciences, Fudan University,
Shanghai 200433, China. Email: {\tt guolina@fudan.edu.cn.}}\quad
Zhiqiang Wang\thanks{School of Mathematical Sciences, Fudan
University, Shanghai 200433, China. Email: {\tt wzq@fudan.edu.cn.}}
}
\date{}
\maketitle

\begin{abstract}
By means of a direct and constructive method based on the theory of
semiglobal $C^2$ solution, the local exact boundary observability is
shown for nonautonomous 1-D quasilinear wave equations. The
essential difference between nonautonomous wave equations and
autonomous ones is also revealed.
\end{abstract}


{\bf 2000 MR Subject Classification:}\quad 35L70, 93B07, 35R30

\section{Introduction}

This paper deals with the following 1-D quasilinear nonautonomous
wave equation
\beq\lab{1.1}
u_{tt}-c^2(t,x,u,u_x,u_t)u_{xx}=f(t,x,u,u_x,u_t), \eeq where $c, f$
are suitably smooth functions with respect to their arguments,
$c=c(t,x,u,u_x,u_t)>0$ is the propagation speed of the nonlinear
wave, and $f$ satisfies \beq\lab{1.2} f(t,x,0,0,0)\equiv 0. \eeq
$u=0$ is an equilibrium of system \e{1.1}. Here we emphasize that
the wave speed $c$ explicitly depends on time which will bring some
new phenomena in the features of observability, as we will see
later.

The boundary conditions can be one of the following physically
meaningful inhomogeneous boundary conditions
\begin{subequations}\label{1.3}
\begin{align}\label{1.3.1}
 &x=0:\ u=h(t),\\\label{1.3.2}
 &x=0:\ u_x=h(t),\\\label{1.3.3}
 &x=0:\ u_x-\alpha u=h(t),\\\label{1.3.4}
 &x=0:\ u_x-\beta u_t=h(t)
\end{align}
\end{subequations}
and a similar one of the following boundary conditions
\begin{subequations}\label{1.4}
\begin{align}\label{1.4.1}
 &x=L:\ u=\bar h(t),\\\label{1.4.2}
 &x=L:\ u_x=\bar h(t),\\\label{1.4.3}
 &x=L:\ u_x+\bar\alpha u=\bar h(t),\\\label{1.4.4}
 &x=L:\ u_x+\bar\beta u_t=\bar h(t),
\end{align}
\end{subequations}
where $\alpha, \beta, \bar\alpha,$ and $\bar\beta$ are positive
constants.

The initial condition is given by
\beq\lab{1.5}
t=t_0:\ (u,u_t)=(\varphi(x), \psi(x)),\quad 0\leq x\leq L,
\eeq
where $(\varphi, \psi)\in C^2[0,L]\times C^1[0,L]$, and the
conditions of $C^2$ compatibility are supposed to be satisfied
at the points $(t_0,0)$ and $(t_0,L)$, respectively.

The \emph{exact observability} problem which we are interested in
can be described as follows: can we find $T>0$ and some suitable
observations $k(t)$ (the value of $u$ or $u_x$ of the solution
$u=u(t,x)$ to the mixed problem \e{1.1} and \e{1.3}-\e{1.5}), such
that the initial data $\varphi$ can be uniquely determined by the
observations $k(t)$ together with the known given boundary functions
$(h(t),\bar h(t))$ on the time interval $[t_0,t_0+T]$? Moreover, can
we have an estimate (\emph{observability inequality}) on $\varphi$
in terms of $k(t)$ and $(h(t),\bar h(t))$? More precisely, noting
that $u=0$ is an equilibrium of system \e{1.1}, we will focus on the
local exact boundary observability for the nonautonomous mixed
problem \e{1.1} and \e{1.3}-\e{1.5} in a neighbourhood (in
$C^2$-sense) of $u=0$.

Exact controllability and observability for wave equations (and
other partial differential equations) have been intensively studied
since Russell \cite{Russell} and Lions \cite{Lions}. Classical
techniques to derive observability estimates for linear wave
equations are mainly the following : \emph{Multiplier Methods} (see
\cite{Komornik,Lions,Osses}), \emph{Carleman Estimates} (see
\cite{DZZ,Iman,Yao,Zhang}), \emph{Microlocal Analysis} (see
\cite{BLR,Burq}), \emph{Spectral Method} (see
\cite{LiuRao,Shubov,Tucsnak}) etc. Due to the duality arguments (see
\cite{Komornik, Lions, Russell, Zuazua}), we know that exact
controllability of a linear system can be reduced to the
observability estimate of its dual system. However, in general, the
duality principle dose not hold for nonlinear dynamical systems (see
\cite{Coron, Libook09}). Consequently, one has to study
controllability and observability for the nonlinear systems
separately. With usual energy estimates and perturbation method,
Pan, Teo and Zhang \cite{PTZ} studied observability (in that paper,
it is called state observation problem) for a semilinear wave
equation, and they also gave a conceptual algorithm of resolution.
Concerning the controllability for nonlinear wave equations, there
are also some results (see \cite{LaTri,Zuazua3, Zuazua4} for
semilinear case, and \cite{ZhouLei} for quasilinear case). For
autonomous 1-D quasilinear wave equations, Li and his collaborators
established a complete theory on exact boundary controllability and
observability, by means of a direct and constructive method which is
based on the theory of semiglobal $C^2$ solution (see
\cite{Libook09,Li,LiYu}).

Up to our knowledge, there are few results on controllability and
observability for nonautonomous wave equations, in which the wave
operator (the principle part of the wave equation) depends
explicitly on time. Cavalcanti \cite{Cavalcanti} established exact
boundary controllability for $n$-D linear nonautonomous wave
equation by utilizing  \emph{Hilbet Uniqueness Method} of Lions
\cite{Lions}. In \cite{Cavalcanti}, the assumption that the wave
speed is larger than a positive constant is vital to obtain the main
results. however, as is pointed out in Section 2, the degenerate
case that the propagation speed approximates zero may produce some
delicate new phenomena in observability (and also in
controllability, see \cite{Wang2}).

In this paper, we establish local exact boundary observability, by
Li's method (with some modification), for some nonautonomous 1-D
quasilinear wave equations, while exact boundary controllability for
these equations has already been established by Wang \cite{Wang2}.
Li's method is said to be a direct and constructive one, since it
treats the quasilinear system directly without any linearization and
fixed point (or compactness) arguments and the observability
inequality can be obtianed by solving some well-posed mixed
problems.  This method is based on the theory of so-called
\emph{semi-global classical solution} (see \cite{Libook09, LiJin,
Wang1}) which guarantees the well-posedness of classical solution on
a preassigned (possibly quite large) time interval $[t_0,t_0+T]$.
Li's method is very useful in 1-D case and it can be used for all
possible linear or nonlinear boundary conditions.

Compared with the results in \cite{Li}, the main difficulties that
we encounter here lie in two parts: to get the existence and
uniqueness of semi-global $C^2$ solution to the nonautonomous mixed
problem \e{1.1} and \e{1.3}-\e{1.5}; and to have a better estimate
on observation time $T$ which is no more as easy as the autonomous
case \cite{Li}. Moreover, we have pay attention to the influence of
the boundary functions $(h(t),\bar h(t))$, while \cite{Li} considers
only the situation that $h\equiv \bar h\equiv 0$. We point out also
that the results obtained in this paper cover all the results
obtained in \cite{Li}.

The organization of this paper is as follows: by a simple example,
we show the possible features of the exact observability for
nonautonomous quasilinear wave equations in Section 2. The
fundamental theory of semiglobal $C^2$ solution to the nonautonomous
mixed problem \e{1.1} and \e{1.3}-\e{1.5} is introduced in Section
3. Adopting Li's method, the main results, Theorems \ref{thm
4.1}-\ref{thm 4.2}, are proved in Section 4. Finally, some remarks
are given in Section 5.

For the convenience of statement, we denote in the whole paper that
\beq\lab{1.6}
l=
\begin{cases}
2 & \text{for } \e{1.3.1}\\1 & \text{for } \e{1.3.2}-\e{1.3.4}
\end{cases}
\eeq
and
\beq\lab{1.7}
\bar l=
\begin{cases}
2 & \text{for } \e{1.4.1}\\1 & \text{for } \e{1.4.2}-\e{1.4.4}.
\end{cases}
\eeq  We also denote C as a positive constant which is independent
of the solution and  C can be different constants in different
situations.

\section{Features of exact observability for nonautonomous
quasilinear wave equations}

The results in \cite{Li} show that, for autonomous quasilinear wave
equations, one can choose proper boundary observed values to
uniquely determine any given small initial value $(\vph, \psi)$ at
$t=0$, provided that the observability time $T>0$ is large enough.
By translation, this conclusion still holds if the observation
starts at the initial time $t=t_0$ instead of $t=0$. Hence, the
observability time $T$ can be chosen to be independent of $t_0$ in
the autonomous case.

In nonautonomous cases, however, generically speaking, the exact
boundary observability should depend on the selection of the initial
time. Consider the linear nonautonomous wave equation
\beq\lab{2.1}
u_{tt}-(c(t))^2u_{xx}=0, \eeq which is a special case of \e{1.1} as
$c$ depends only on time $t$. One can see that:

1) the two-sides exact boundary observability holds for \e{2.1} on
the time interval $[t_0,t_0+T]$ if and only if $\int_{t_0}^{t_0+T}
c(t)dt \geq L$;

2) the one-side exact boundary observability holds for \e{2.1} on
the time interval $[t_0,t_0+T]$ if and only if $\int_{t_0}^{t_0+T}
c(t) dt \geq 2L$.\\
By the different choices of $c(t)$, it is easy to see that there are
three possibilities: the exact boundary observability for \e{2.1}
holds

1) only for some initial time $t_0\in \mathbb{R}$, but not for the
others;

2) for none of the initial time $t_0\in \mathbb{R}$;\\ or

3) for all the initial time $t_0\in \mathbb{R}$. \\
However, there is only the possibility 3) in autonomous case as
shown by Remark \ref{rem 5.5}.

Moreover, in general the observability time $T$ for \e{2.1} depends
on the initial time $t_0$, that is to say, the exact boundary
observability holds only when $T>T(t_0)$. On the other hand, the
observability time $T$ might be independent of $t_0$ in some special
cases, for instance, if $c(t)$ is a suitable periodic function.

Thus, the exact boundary observability for nonautonomous hyperbolic
systems is much more complicated than that in autonomous cases,
and we should pay more attention on it.

\section{Semiglobal $C^2$ solution to 1-D nonautonomous quasilinear wave equations}

In this section we shall establish the theory on the semiglobal $C^2$
solution to the mixed initial-boundary value problem \e{1.1} and
\e{1.3}-\e{1.5} on the domain
\beq\lab{3.1}
R(t_0,T_0)=\{(t,x)|t_0\leq t \leq t_0+T_0, 0\leq x \leq L\},
\eeq
where $T_0>0$ is a preassigned and possibly quite large number.

Suppose that the conditions of $C^2$ compatibility are satisfied at
the points $(t,x)=(t_0,0)$ and $(t_0,L)$, respectively. In order to
get the semiglobal $C^2$ solution to the mixed problem for \e{1.1}
with various kinds of boundary conditions in a unified manner, we
reduce the problem to a corresponding mixed problem for a first
order quasilinear hyperbolic system (cf. \cite{LiYu,Wang2}).

Let
\beq\lab{3.2}
v=u_x,\quad w=u_t
\eeq
and
\beq\lab{3.3}
U=(u,v,w)^T.
\eeq
\e{1.1} can be reduced to the following first order quasilinear system
\beq\lab{3.4}
\begin{cases}
u_t=w,\\ v_t-w_x=0,\\w_t-c^2(t,x,u,v,w)v_x=f(t,x,u,v,w).
\end{cases}
\eeq
Accordingly, the initial condition \e{1.5} reduces to
\beq\lab{3.5}
t=t_0: \ \ U=(\vph(x),\vph'(x),\psi(x))^T,\quad 0\leq x\leq L.
\eeq
Since $c(t,x,u,u_x,u_t)>0$, \e{3.4} is a strictly hyperbolic system with three
distinct real eigenvalues
\beq\lab{3.6}
\lambda_1=-c(t,x,u,v,w)<\lambda_2\equiv 0<\lambda_3=c(t,x,u,v,w),
\eeq
and the corresponding left eigenvectors can be taken as
\beq\lab{3.7}
\begin{cases}
l_1(t,x,U)=(0,c(t,x,u,v,w),1),\\ l_2(t,x,U)=(1,0,0),\\ l_3(t,x,U)=(0,-c(t,x,u,v,w),1).
\end{cases}
\eeq

Setting
\beq\lab{3.8}
v_i=l_i(t,x,U)U\quad\i13,
\eeq
namely,
\beq\lab{3.9}
\begin{cases}
v_1=c(t,x,u,v,w)v+w,\\ v_2=u,\\ v_3=-c(t,x,u,v,w)v+w,
\end{cases}
\eeq
we have
\beq\lab{3.10}
\begin{cases}
v_1+v_3=2w,\\ v_1-v_3=2c(t,x,u,v,w)v.
\end{cases}
\eeq

The boundary condition \e{1.3.1} can be rewritten as \beq\lab{3.11}
x=0:\ \ v_1+v_3=2h'(t) \eeq together with the following condition of
$C^0$ compatibility \beq\lab{3.12} h(t_0)=\vph (0). \eeq In a
neighborhood of $U=0$, the boundary conditions \e{1.3.2}-\e{1.3.3}
can be equivalently rewritten as
\begin{align}\lab{3.13}
x=0:& \quad v_3=p_2(t,v_1,v_2)+q_2(t),\\ \lab{3.14}
x=0:& \quad v_3=p_3(t,v_1,v_2)+q_3(t),
\end{align}
or
\begin{align}\lab{3.15}
x=0:& \quad v_1=\tilde p_2(t,v_2,v_3)+\tilde q_2(t),\\ \lab{3.16}
x=0:& \quad v_1=\tilde p_3(t,v_2,v_3)+\tilde q_3(t).
\end{align}
Similarly, in a neighborhood of $U=0$, the boundary condition \e{1.3.4} can
be rewritten as
\beq\lab{3.17}
x=0:\ \ v_3=p_4(t,v_1,v_2)+q_4(t)
\eeq
or, when
\beq\lab{3.18}
\beta \neq \f {1}{c(t,0,0,0,0)},\quad \forall t \in [t_0,t_0+T_0],
\eeq
\beq\lab{3.19}
x=0:\quad v_1=\tilde p_4(t,v_2,v_3)+\tilde q_4(t).
\eeq

Moreover, we have
\beq\lab{3.20}
p_i(t,0,0)\equiv \tilde p_i(t,0,0)\equiv 0\quad (i=2,3,4),
\eeq
and
\beq\nonumber
\|q_i\|_{C^1[t_0,t_0+T_0]},\|\tilde q_i\|_{C^1[t_0,t_0+T_0]}\rightarrow 0
\quad \2i4
\eeq
as $\|h\|_{C^1[t_0,t_0+T_0]}\rightarrow 0$.

Similarly, the boundary condition \e{1.4.1} can be rewritten as
\beq\lab{3.21}
x=L:\quad v_1+v_3=2\bar h'(t)
\eeq
together with
\beq\lab{3.22}
\bar h(t_0)=\vph (L).
\eeq
\e{1.4.2}-\e{1.4.4} can be rewritten as
\beq\lab{3.23}
x=L:\quad v_1=p(t,v_2,v_3)+q(t).
\eeq
Moreover, when
\beq\lab{3.24}
\bar\beta\neq\f {1}{c(t,L,0,0,0)},\quad\forall t\in [t_0,t_0+T_0],
\eeq
\e{3.23} can be equivalently rewritten as
\beq\lab{3.25}
x=L:\quad v_3=\tilde p(t,v_1,v_2)+\tilde q(t),
\eeq
where
\beq\lab{3.26}
p(t,0,0)\equiv \tilde p(t,0,0)\equiv 0,
\eeq
and
\beq\nonumber
\|q\|_{C^1[t_0,t_0+T_0]},\|\tilde q\|_{C^1[t_0,t_0+T_0]}\rightarrow 0
\eeq
as $\|\bar h\|_{C^1[t_0,t_0+T_0]}\rightarrow 0$.

Obviously, the conditions of $C^2$ compatibility at the points
$(t_0,0)$ and $(t_0,L)$ for the mixed problem \e{1.1}
and\e{1.3}-\e{1.5} guarantee the conditions of $C^1$ compatibility
for the corresponding mixed problem of the first order quasilinear
hyperbolic system \e{3.4}-\e{3.5},\e{3.11} (or \e{3.15} or \e{3.16}
or \e{3.17}) and \e{3.21} (or \e{3.23}).

Applying the theory on the semiglobal $C^1$ solution to the mixed
initial-boundary value problem of first order nonautonomous
quasilinear hyperbolic systems (cf. \cite{Wang1}), we get

\begin{lem}{\bf (Semiglobal $C^2$ solution)}\lab{lem 3.1}
Suppose that $c,f\in C^1, c>0$ and \e{1.2} holds. Suppose furthermore
that $\vph\in C^2,$ $\psi\in C^1,$ $h\in C^l,\bar h\in C^{\bar l}$
$($see \e{1.6}-\e{1.7}$)$ and the conditions of $C^2$ compatibility are
supposed to be satisfied at the points $(t_0,0)$ and $(t_0,L)$
respectively. For any given $T_0>0$ (possibly quite large), if
$\|(\vph,\psi)\|_{C^2[0,L]\times C^1[0,L]}$ and
$\|(h,\bar h)\|_{C^l[t_0,t_0+T_0] \times C^{\bar l}[t_0,t_0+T_0]}$
are sufficiently small (depending on $t_0$ and $T_0$), the mixed problem
\e{1.1} and \e{1.3}-\e{1.5} admits a unique $C^2$ solution $u=u(t,x)$
(called semiglobal $C^2$ solution) on the domain
$R(t_0,T_0)\triangleq\{(t,x)|t_0\leq t\leq t_0+T_0,\ 0\leq x\leq L\}$,
and the following estimate holds
\beq\lab{3.29}
\|u\|_{C^2[R(t_0,T_0)]}\leq C(\|(\vph,\psi)\|_{C^2[0,L]\times C^1[0,L]}
+\|(h,\bar h)\|_{C^l[t_0,t_0+T_0]\times C^{\bar l}[t_0,t_0+T_0]}).
\eeq
\end{lem}

\begin{cor}\lab{cor 3.1}
Suppose that $c,f\in C^1, c>0$ and \e{1.2} holds. If
$\|(\vph,\psi)\|_{C^2[0,L]\times C^1[0,L]}$ is sufficiently small,
then Cauchy problem \e{1.1} and \e{1.5} admits a unique global $C^2$
solution $u=u(t,x)$ on the whole maximum determinate domain
$D=\{(t,x)|t\geq t_0, x_1(t)\leq x\leq x_2(t)\}$ (see \cite{Li}),
where the two curves $x_1(t),x_2(t)$ are defined as follows: \beq
\begin{cases}
\displaystyle\f {dx_1}{dt}=c(t,x_1,u(t,x_1),u_x(t,x_1),u_t(t,x_1)),\\
t=t_0:\ x_1=0
\end{cases}
\eeq and  \beq
\begin{cases}
\displaystyle\f {dx_2}{dt}=-c(t,x_2,u(t,x_2),u_x(t,x_2),u_t(t,x_2)),\\
t=t_0:\ x_2=L,
\end{cases}
\eeq respectively. Moreover, we have the following estimate
\beq\lab{3.30} \|u\|_{C^2[D]}\leq C\|(\vph,\psi)\|_{C^2[0,L]\times
C^1[0,L]}. \eeq
\end{cor}

\vskip 5mm

\begin{center}
\scriptsize \setlength{\unitlength}{1mm}
\begin{picture}(50,45)
\linethickness{1pt} \put(5,3){\vector(1,0){48}}
\put(10,-5){\vector(0,1){50}} \put(45,-2){\line(0,1){42}}
\qbezier(25,18)(20,5)(10,3) \qbezier(25,18)(35,6)(45,3)
\put(5,4){$t_0$} \put(5,40){$t$} \put(13,10){$x_1(t)$}
\put(34,10){$x_2(t)$} \put(11,-1){$0$} \put(41,-1){$L$}
\put(50,-1){$x$} \put(24,6){$D$}
\end{picture}
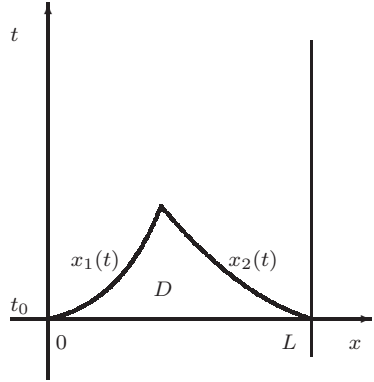
\vskip5mm  Figure 1. Maximum determinate domain $D$ of Cauchy
problem
\end{center}

\vskip 5mm

\section{Exact boundary observability for 1-D nonautonomous quasilinear wave equations}

Now we consider the exact boundary observability for system \e{1.1}
and \e{1.3}-\e{1.5}. Let $t_0$ be the initial time and let $T$ be
the observability time. Define the domain $R(t_0,T)$ similar to
\e{3.1}.

The principle of choosing the observed value is that the observed
value together with the boundary condition can uniquely determine
the value $(u,u_x)$ on the boundary (cf. \cite{Li}). Hence, the
observed value at $x=0$ can be taken as
\begin{subequations}\label{4.2}
\begin{align}\nonumber
 &1.\ u_x=k(t) \ \text{for \e{1.3.1}, then}\\\label{4.2.1}
 &\quad\quad\quad\quad\quad\quad\quad\quad x=0:\quad (u,u_x)=(h(t),k(t)),
 \\\nonumber
 &2.\ u=k(t) \ \text{for \e{1.3.2}, then}\\\label{4.2.2}
 &\quad\quad\quad\quad\quad\quad\quad\quad x=0:\quad (u,u_x)=(k(t),h(t)),
 \\\nonumber
 &3.\ u=k(t) \ \text{for \e{1.3.3}, then}\\\label{4.2.3}
 &\quad\quad\quad\quad\quad\quad\quad\quad x=0:\quad (u,u_x)=(k(t),\alpha k(t)+h(t)),
 \\\nonumber
 &4.\ u=k(t)\  \text{for \e{1.3.4}, then}\\\label{4.2.4}
 &\quad\quad\quad\quad\quad\quad\quad\quad x=0:\quad (u,u_x)=(k(t),\beta k'(t)+h(t)).
\end{align}
\end{subequations}
Then, by means of the observed value at $x=0$, we get
\beq\lab{4.3}
x=0:\ (u,u_x)=(a(t),b(t)),
\eeq
and for any given $T$,
\beq\lab{4.4}
\|(a,b)\|_{C^2[t_0,t_0+T]\times C^1[t_0,t_0+T]}\leq C(\|k\|_{C^d[t_0,t_0+T]}
+\|h\|_{C^l[t_0,t_0+T]}),
\eeq
where $l$ is given by \e{1.6} and
\beq\lab{4.5}
d=
\begin{cases}
1 & \text{for }\e{1.3.1},\\
2 & \text{for }\e{1.3.2}-\e{1.3.4}.
\end{cases}
\eeq

The observed value $\bar k(t)$ at $x=L$ can be similarly taken,
then we get
\beq\lab{4.6}
x=L:\ (u,u_x)=(\bar a(t),\bar b(t)),
\eeq
and for any given $T$,
\beq\lab{4.7}
\|(\bar a,\bar b)\|_{C^2[t_0,t_0+T]\times C^1[t_0,t_0+T]}\leq
C(\|\bar k\|_{C^{\bar d}[t_0,t_0+T]}+\|\bar h\|_{C^{\bar l}[t_0,t_0+T]}),
\eeq
where $\bar l$ is given by \e{1.7} and
\beq\lab{4.8}
\bar d=
\begin{cases}
1 & \text{for }\e{1.4.1},\\
2 & \text{for }\e{1.4.2}-\e{1.4.4}.
\end{cases}
\eeq

\begin{thm}{\bf (Two-sides observation)}\lab{thm 4.1}
Suppose that $c,f\in C^1, c>0$ and \e{1.2} holds. Suppose furthermore
that there exists $T>0$ such that
\beq\lab{4.9}
\int_{t_0}^{t_0+T}\inf_{0 \leq x \leq L}c(t,x,0,0,0)\ dt>L. \eeq For
any given initial data $(\phi,\psi)$ with
$\|(\vph,\psi)\|_{C^2[0,L]\times C^1[0,L]}$ to be suitably small,
suppose finally that the conditions of $C^2$ compatibility are
satisfied at the points $(t_0,0)$ and $(t_0,L)$ respectively. Then
the initial data $(\phi,\psi)$ can be uniquely determined by the
observed values $k(t)$ at $x=0$ and $\bar k(t)$ at $x=L$ together
with the known boundary functions $(h(t),\bar h(t))$ on the interval
$[t_0,t_0+T]$. Moreover, the following observability inequality
holds: \beq\lab{4.10} \|(\vph,\psi)||_{C^2[0,L]\times C^1[0,L]}\leq
C(\|(k,\bar k)\|_{C^d[t_0,t_0+T]\times C^{\bar d}[t_0,t_0+T]}
+\|(h,\bar h)\|_{C^l[t_0,t_0+T]\times C^{\bar l}[t_0,t_0+T]}), \eeq
where $d, \bar d, l$ and $\bar l$ are given by \e{4.5}, \e{4.8},
\e{1.6} and \e{1.7} respectively.
\end{thm}

\begin{proof}
Noting \e{4.9}, there exists $\varepsilon>0$ such that
\beq\lab{4.11}
\int_{t_0}^{t_0+T}\inf_{0 \leq x \leq L \atop
{|(u,v,w)|\leq\varepsilon}} c(t,x,u,v,w)\ dt>L, \eeq in which
$|(u,v,w)|=\sqrt{|u|^2+|v|^2+|w|^2}$.

By Lemma \ref{lem 3.1}, when $\|(\vph,\psi)\|_{C^2[0,L]\times
C^1[0,L]}$ and $\|(h,\bar h)\|_{C^l[t_0,t_0+T]\times C^{\bar
l}[t_0,t_0+T]}$ are sufficiently small, the mixed problem \e{1.1}
and \e{1.3}-\e{1.5} admits a unique semiglobal $C^2$ solution
$u=u(t,x)$ with small $C^2$ norm on the domain $R(t_0,T)$. Hence, the
$C^d$ and $C^{\bar d}$ norm of the observed value $k(t)$ and $\bar k(t)$
are sufficiently small respectively. In particular, we may suppose
\beq\lab{4.12}
\|u\|_{C^1[R(t_0,T)]}\leq\varepsilon.
\eeq

Noting $c>0$, we can change the role of $t$ and $x$ in equation
\e{1.1} in order to solve it in the $x-$direction.

By Corollary \ref{cor 3.1}, the rightward Cauchy problem for
equation \e{1.1} with the initial condition \e{4.3} admits a unique
$C^2$ solution $u=\tilde u(t,x)$ on the whole maximum determinate
domain $D_r$ and \beq\lab{4.13} \|\tilde u\|_{C^2[D_r]}\leq
C\big(\|k\|_{C^d [t_0,t_0+T]}+\|h\|_{C^l[t_0,t_0+T]} \big). \eeq
Here $D_r=\{(t,x)|t_0 \leq t\leq t_0+T, 0\leq x\leq
\min\{x_1(t),x_2(t)\}\}$, in which the two curves $x_1(t),x_2(t)$
are defined as following \beq\lab{4.15}
\begin{cases}
\dis\f {dx_1}{dt}=-c(t,x_1,u(t,x_1),u_x(t,x_1),u_t(t,x_1)),\\
t=t_0+T:\ x_1=0
\end{cases}
\eeq and \beq\lab{4.16}
\begin{cases}
\displaystyle\f {dx_2}{dt}=c(t,x_2,u(t,x_2),u_x(t,x_2),u_t(t,x_2)),\\
t=t_0:\ x_2=0.
\end{cases}
\eeq

Similarly, the leftward Cauchy problem for equation \e{1.1} with the
initial condition \e{4.6} admits a unique $C^2$ solution
$u=\tilde{\tilde u}(t,x)$ on the whole maximum determinate domain
$D_l$ and
\beq\lab{4.14}
\|\tilde{\tilde u}\|_{C^2[D_l]}\leq C\big(\|\bar k\|_{C^{\bar d}
[t_0,t_0+T]} +\|\bar h\|_{C^{\bar l}[t_0,t_0+T]} \big). \eeq Here
$D_l=\{(t,x)|t_0 \leq t\leq t_0+T, \max\{x_3(t),x_4(t)\}\leq x\leq L
\}$, in which the two curves $x_3(t),x_4(t)$ are defined as
following \beq\lab{4.17}
\begin{cases}
\displaystyle\f {dx_3}{dt}=c(t,x_3,u(t,x_3),u_x(t,x_3),u_t(t,x_3)),\\
t=t_0+T:\ x_3=L
\end{cases}
\eeq and \beq\lab{4.18}
\begin{cases}
\displaystyle\f {dx_4}{dt}=-c(t,x_4,u(t,x_4),u_x(t,x_4),u_t(t,x_4)),\\
t=t_0:\ x_4=L.
\end{cases}
\eeq

We now claim that the domains $D_r$ and $D_l$ must intersect each
other.

Since $x_1=x_1(t)$ passes through the point $(t_0+T,0)$, it follows
from \e{4.15} that \beq\lab{4.19}
x_1(t)=\int_t^{t_0+T}c(t,x_1,u(t,x_1),u_x(t,x_1),u_t(t,x_1)) dt.
\eeq Hence, noting \e{4.11}-\e{4.12}, the intersection point of
$x=x_1(t)$ with the line $x=L$ must be above the point $(t_0,L)$,
where $x=x_4(t)$ passes through. Noting that the ODE in \e{4.15} is
the same as that in \e{4.18}, we conclude by the uniqueness of $C^1$
solution that $x=x_1(t)$ stays above $x=x_4(t)$ all the time.
Similarly, $x=x_3(t)$ always stays above $x=x_2(t)$. Thus $D_r$ and
$D_l$ intersect each other.

Therefore, there exists $\tilde T\in (t_0,t_0+T)$ such that the value
$(u,u_t)=(\Phi (x),\Psi (x))$ on $t=\tilde T$ can be completely determined
by $u=\tilde u(t,x)$ and $u=\tilde{\tilde u}(t,x)$. Then we
get from \e{4.13} and \e{4.14} that
\beq\lab{4.20}
\|(\Phi,\Psi)\|_{C^2[0,L]\times C^1[0,L]}\leq
C\big(\|(k,\bar k)\|_{C^d[t_0,t_0+T]\times C^{\bar d} [t_0,t_0+T]}
+\|(h,\bar h)\|_{C^l [t_0,t_0+T]\times C^{\bar l}[t_0,t_0+T]} \big).
\eeq

Since both $u=\tilde u(t,x)$ and $u=\tilde{\tilde u}(t,x)$ are the
restriction of the $C^2$ solution $u=u(t,x)$ to the original mixed problem
\e{1.1} and \e{1.3}-\e{1.5} on the corresponding maximum determinate domains
respectively, we have
\beq\lab{4.21}
t=\tilde T:\ u=\Phi (x), u_t=\Psi(x),\quad 0\leq x\leq L.
\eeq

By Lemma \ref{lem 3.1}, the backward mixed initial-boundary value problem
\e{1.1} with the initial condition \e{4.21} and the boundary conditions
\begin{align}\lab{4.22}
x=0:&\quad u=a(t),\\\lab{4.23}
x=L:&\quad u=\bar a(t)
\end{align}
admits a unique semiglobal $C^2$ solution $u=\hat u(t,x)$ on
\beq\lab{4.24}
R(t_0,\tilde T)=\{(t,x)|t_0\leq t\leq t_0+\tilde T, 0\leq x\leq L\},
\eeq
since the conditions of $C^2$ compatibility at the points $(t,x)=(\tilde T,0)$
and $(\tilde T, L)$ are obviously satisfied respectively. By the uniqueness
of $C^2$ solution, $u=\hat u(t,x)$ must be the restriction of the original
$C^2$ solution $u=u(t,x)$ on $R(t_0,\tilde T)$, and the following estimate
holds:
\beq\lab{4.25}
\|u\|_{C^2[R(t_0,\tilde T)]}\leq C\big(\|(\Phi,\Psi)\|_{C^2[0,L]\times C^1[0,L]}
+\|(a,\bar a)\|_{C^2[t_0,t_0+T]\times C^2[t_0,t_0+T]}\big ).
\eeq

\vspace{1cm}

\begin{minipage}{80mm}
\begin{center}
\scriptsize \setlength{\unitlength}{1mm}
\begin{picture}(50,45)
\linethickness{1pt} \put(5,3){\vector(1,0){41}}
\put(10,-5){\vector(0,1){54}} \put(10,25){\vector(1,0){5}}
\put(38,15){\vector(-1,0){5}} \put(38,-2){\line(0,1){46}}
\multiput(5,18)(2,0){20}{\line(1,0){1.5}} \put(5,40){\line(1,0){41}}
\qbezier(30,23)(20,8)(10,3) \qbezier(10,40)(18,26)(30,23)
\qbezier(38,40)(33,28)(20,15) \qbezier(20,15)(30,5)(38,3)
\put(0,41){$t_0+T$} \put(5,4){$t_0$} \put(5,19){$\widetilde T$}
\put(7,49){$t$} \put(11,-1){$0$} \put(35,-1){$L$} \put(45,-1){$x$}
\put(14,20){$D_r$} \put(28,14){$D_l$} \put(31,34){$x_3$}
\put(32,6){$x_4$} \put(16,34){$x_1$} \put(15,9){$x_2$}
\end{picture}

\vskip5mm  Figure 2. $D_r$ and $D_l$ intersect each other

\end{center}
\end{minipage}
\begin{minipage}{50mm}
\begin{center}
\scriptsize \setlength{\unitlength}{1mm}
\begin{picture}(60,45)
\linethickness{1pt} \put(5,3){\vector(1,0){41}}
\put(10,-5){\vector(0,1){54}} \put(24,18){\vector(0,-1){5}}
\put(38,-2){\line(0,1){46}} \put(5,18){\line(1,0){41}}
\put(5,40){\line(1,0){41}} \put(0,41){$t_0+T$} \put(5,4){$t_0$}
\put(5,19){$\widetilde T$} \put(7,49){$t$} \put(11,-1){$0$}
\put(35,-1){$L$} \put(45,-1){$x$} \put(20,20){$(\Phi,\Psi)$}
\put(18,8){$R(t_0,\widetilde T)$}
\end{picture}

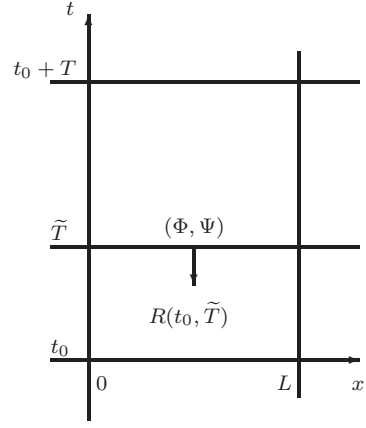
\vskip5mm  Figure 3. Solve the backward problem on $R(t_0,\widetilde
T)$

\end{center}
\end{minipage}

\vskip 5mm

Finally, \e{4.10} follows immediately from \e{1.5}, \e{4.4}, \e{4.7}
and \e{4.20}. The concludes the proof of Theorem \ref{thm 4.1}.
\end{proof}

\begin{thm}{\bf (One-side observation)}\lab{thm 4.2}
Under the assumptions of Theorem \ref {thm 4.1} $($except \e{4.9}$)$, suppose
furthermore that $\bar\beta$ in the boundary condition \e{1.4.4} satisfies
\beq\label{4.27}
\bar\beta\neq\f {1}{c(t,L,0,0,0)},\quad\forall t\in [t_0,t_0+T]
\eeq
and there exists $T>0$ such that
\beq\label{4.26}
\int_{t_0}^{t_0+T}\inf_{0 \leq x \leq L}c(t,x,0,0,0)\ dt>2L.
\eeq
Then, the initial data $(\varphi, \psi)$ can be uniquely determined by
the observed value $k(t)$ at $x=0$ together with the known boundary
functions $(h(t),\bar h(t))$ on the interval $[t_0,t_0+T]$. Moreover,
the following observability inequality holds:
\beq\lab{4.28}
\|(\vph,\psi)||_{C^2[0,L]\times
C^1[0,L]}\leq C(\|k\|_{C^d[t_0,t_0+T]} +\|(h,\bar
h)\|_{C^l[t_0,t_0+T]\times C^{\bar l}[t_0,t_0+T]}).
\eeq
\end{thm}

\begin{proof}
Similarly to the proof of Theorem \ref{thm 4.1}, the rightward
Cauchy problem for equation \e{1.1} with the initial condition
\e{4.3} admits a unique $C^2$ solution $u=\tilde u(t,x)$ on the
whole maximum determinate domain $D_r$ and estimate \e{4.13} holds.
Under assumption \e{4.26}, $D_r$ must intersect the line $x=L$.

Thus, there exists $\widetilde T\in (t_0,t_0+T)$ such that the value
$(u,u_t)=(\Phi(x),\Psi(x))$ on $t=\widetilde T$ can be completely
determined by $u=\tilde u(t,x)$. Then, we get from \e{4.13} that
\beq\lab{4.29} \|(\Phi,\Psi)\|_{C^2[0,L]\times C^1[0,L]}\leq
C\big(\|k\|_{C^d[t_0,t_0+T]}+\|h\|_{C^l [t_0,t_0+T]}\big). \eeq

Since the conditions of $C^2$ compatibility at the points
$(t,x)=(\widetilde T,0)$ and $(\widetilde T, L)$ are obviously
satisfied respectively, by Lemma \ref{lem 3.1}, the backward mixed
problem \e{1.1} with the initial condition \e{4.21} and the boundary
conditions \e{1.4} and \beq\lab{4.30} x=0:\quad u=a(t) \eeq admits a
unique $C^2$ solution $u=\hat u(t,x)$ on $R(t_0,\widetilde T)$. By
the uniqueness of solution, $u=\hat u(t,x)$ must be the restriction
of the original $C^2$ solution $u=u(t,x)$ on $R(t_0,\widetilde T)$,
and the following estimate holds: \beq\lab{4.31}
\|u\|_{C^2[R(t_0,\widetilde T)]}\leq
C\big(\|(\Phi,\Psi)\|_{C^2[0,L]\times C^1[0,L]} +\|(a,\bar
h)\|_{C^2[t_0,t_0+T]\times C^{\bar l}[t_0,t_0+T]}\big). \eeq

\vskip 5mm

\begin{minipage}{80mm}
\begin{center}
\scriptsize \setlength{\unitlength}{1mm}
\begin{picture}(50,55)
\linethickness{1pt} \put(5,3){\vector(1,0){41}}
\put(10,-5){\vector(0,1){54}} \put(10,18){\vector(1,0){5}}
\put(38,-2){\line(0,1){46}}
\multiput(5,22)(2,0){20}{\line(1,0){1.5}} \put(5,40){\line(1,0){41}}
\qbezier(43,23)(20,4)(10,3) \qbezier(10,40)(18,26)(43,23)
\put(0,41){$t_0+T$} \put(5,4){$t_0$} \put(5,24){$\widetilde T$}
\put(7,49){$t$} \put(11,-1){$0$} \put(35,-1){$L$} \put(45,-1){$x$}
\put(19,18){$D_r$}
\end{picture}

\vskip5mm  Figure 4. $D_r$ intersects $x=L$

\end{center}
\end{minipage}
\begin{minipage}{50mm}
\begin{center}
\scriptsize \setlength{\unitlength}{1mm}
\begin{picture}(60,55)
\linethickness{1pt} \put(5,3){\vector(1,0){41}}
\put(10,-5){\vector(0,1){54}} \put(24,22){\vector(0,-1){5}}
\put(38,-2){\line(0,1){46}} \put(5,22){\line(1,0){41}}
\put(5,40){\line(1,0){41}} \put(0,41){$t_0+T$} \put(5,4){$t_0$}
\put(5,24){$\widetilde T$} \put(7,49){$t$} \put(11,-1){$0$}
\put(35,-1){$L$} \put(45,-1){$x$} \put(20,24){$(\Phi,\Psi)$}
\put(18,12){$R(t_0,\widetilde T)$}
\end{picture}

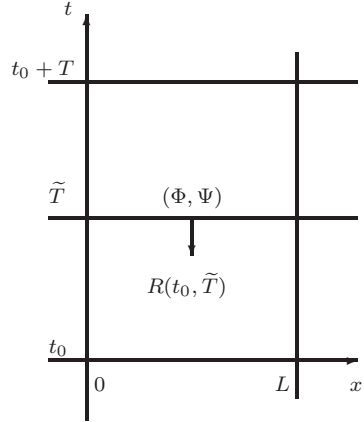
\vskip5mm  Figure 5. Solve backward problem on $R(t_0,\widetilde T)$

\end{center}
\end{minipage}

\vskip 5mm

Noting \e{1.5}, \e{4.28} follows immediately from \e{4.4} and
\e{4.29}. This finishes the proof of Theorem \ref{thm 4.2}.
\end{proof}

\section{Remarks}

\begin{rem}\lab{rem 4.1}
In Theorem \ref{thm 4.1}, \e{4.9} is a sharp estimate on the
observability time $T$, which guarantees that two maximum
determinate domains $D_r$ and $D_l$ intersect each other. In Theorem
\ref{thm 4.2}, \e{4.26} is a sharp estimate on the observability
time $T$, which guarantees that the maximum determinate domain $D_r$
of the rightward Cauchy problem must intersect the line $x=L$. The
assumptions \e{4.9} and \e{4.26} on the observation time allow the
propagation speed $c$ to be close to zero, which is not the case in
\cite{Cavalcanti} even if the wave equation \e{1.1} is linear, i.e.,
$c=c(t,x)$.
\end{rem}

\begin{rem}\lab{rem 4.2}
In Theorem \ref{thm 4.2}, if the observed value $\bar k(t)$ is
chosen at $x=L$ and we assume \beq\lab{4.32}
\beta\neq\f{1}{c(t,0,0,0,0)},\quad\forall t\in [t_0,t_0+T] \eeq
instead of \e{4.27}, a similar result can be obtained.
\end{rem}

\begin{rem}\lab{rem 4.4}
Consider the $n$-dimensional quasilinear wave equation with rotation
invariance \beq\lab{4.33} u_{tt}-c^2(t,|x|,u,u_t,\,x\cdot\nabla
u)\,\Delta u=f(t,|x|,u,u_t,\,x\cdot\nabla u) \eeq on the hollow ball
\beq\lab{4.34}
D=\Big\{x=(x_1,\cdots,x_n)\in \mathbb{R}^n \Big|\ r_1 \leq |x|\leq
r_2,\ |x|=\Big(\sum_{i=1}^n x_i^2\Big)^\f12\Big\}
\quad(0<r_1<r_2).
\eeq

Under the assumption of spherical symmetry, \e{4.33} can be reduced
to the following 1-D nonautonomous wave equation \beq\lab{4.35}
u_{tt}-c^2(t,r,u,u_t,\,ru_r)\,u_{rr}
=f(t,r,u,u_t,ru_r)+(\f{n-1}{r})c^2(t,r,u,u_t,\,ru_r)\,u_r, \eeq
where $r=|x|$, then we can apply Theorems \ref{thm 4.1}-\ref{thm
4.2} directly to obtain the corresponding exact boundary
observability with spherical symmetry data.
\end{rem}

\begin{rem}\label{rem 5.5}

Different from the nonautonomous case, the exact boundary
observability can be always realized for the 1-D essential
autonomous quasilinear wave equation \beq\lab{4.36}
u_{tt}-c^2(x,u,u_x,u_t)\,u_{xx}=f(t,x,u,u_x,u_t), \eeq provided that
the observability time $T$ is large enough. In fact, by Theorems
\ref{thm 4.1}-\ref{thm 4.2},  two-sides (resp., one-side) exact
boundary observability  for \e{4.36} can be realized on the interval
$[t_0,t_0+T]$ if \beq T>\sup_{0 \leq x \leq L}\f {L}{\
c(x,0,0,0)}\quad  (resp., \quad T>\sup_{0 \leq x \leq L}\f {2L}{\
c(x,0,0,0)}). \eeq

\end{rem}

\section*{Acknowledgements}
The authors would like to thank Professor Tatsien Li and the
referees for their valuable suggestions and comments. This work was
partially supported by the Natural Science Foundation of China
(Grant Number: 10701028).


\end{document}